\documentclass{article}
     \PassOptionsToPackage{numbers, compress}{natbib}
\usepackage{arxiv}
\usepackage[utf8]{inputenc} % allow utf-8 input
\usepackage[T1]{fontenc}    % use 8-bit T1 fonts
\usepackage{hyperref}       % hyperlinks
\usepackage{url}            % simple URL typesetting
\usepackage{booktabs}       % professional-quality tables
\usepackage{amsfonts}       % blackboard math symbols

\usepackage{lipsum}
\usepackage{amsfonts,amssymb,amsmath,amsthm}
\usepackage{graphicx}
\usepackage{multicol}
\usepackage{color}
\usepackage{tcolorbox}
\usepackage{epstopdf}
\usepackage[algo2e]{algorithm2e}
\usepackage{algorithmic}
\usepackage{algorithm}
\usepackage{tikz}
\usepackage{subcaption}
\usepackage{mathtools}

\usetikzlibrary{shapes,arrows}

\newtheorem{theorem}{Theorem}
\newtheorem{lemma}{Lemma}
\newtheorem{proposition}{Proposition}
\input{mysymbol.sty}
\def\Binv{(\bbB^t)^{-1}}

\newtheorem{assumption}{\hspace{0pt}\bf Assumption}

\title{DAve-QN: A Distributed Averaged Quasi-Newton Method
with Local Superlinear Convergence Rate}

\author{
  Saeed Soori \\
  ECE Department\\
  Rutgers University\\
  \texttt{saeed.soori@rutgers.edu} \\
  \And
   Konstantin Mischenko \\
   CS Department\\
   KAUST University  \\
   \texttt{konstantin.mishchenko@kaust.edu.sa} \\
   \And
   Aryan Mokhtari \\
   Laboratory for Information and Decision Systems \\
   Massachusetts Institute of Technology \\
   \texttt{aryanm@mit.edu} \\
   \And
   Maryam Mehri Dehnavi \\
   CS Department \\
   University of Toronto  \\
   \texttt{mmehride@cs.toronto.edu} \\
   \And
   Mert Gurbuzbalaban \\
   MSIS Department \\
   Rutgers University \\
   \texttt{mert.gurbuzbalaban@rutgers.edu} \\
}

\begin{document}

\maketitle

\begin{abstract}
  In this paper, we consider distributed algorithms for solving the empirical risk minimization problem under the master/worker communication model. We develop a distributed asynchronous quasi-Newton algorithm that can achieve superlinear convergence. To our knowledge, this is the first distributed asynchronous algorithm with superlinear convergence guarantees. Our algorithm is communication-efficient in the sense that at every iteration the master node and workers communicate vectors of size $O(p)$, where $p$ is the dimension of the decision variable.  The proposed method is based on a distributed asynchronous averaging scheme of decision vectors and gradients in a way to effectively capture the local Hessian information of the objective function. Our convergence theory supports asynchronous computations subject to both bounded delays and unbounded delays with a bounded time-average. Unlike in the majority of asynchronous optimization literature, we do not require choosing smaller stepsize when delays are huge. We provide numerical experiments that match our theoretical results and showcase significant improvement comparing to state-of-the-art distributed algorithms.
\end{abstract}

%!TEX root = ex_article_v2.tex
\vspace{-0.2in}
\section{Introduction} 
Many optimization problems in machine learning including \textit{empirical risk minimization} are based on processing large amounts of data as an input. Due to the advances in sensing technologies and storage capabilities the size of the data we can collect and store increases at an exponential manner. As a consequence, a single machine (processor) is typically not capable of processing and storing all the samples of a dataset. To solve such ``big data'' problems, we typically rely on distributed architectures where the data is distributed over several machines that reside on a communication network \cite{bertsekas1989parallel,recht2011hogwild}. In such modern architectures, the cost of communication is typically orders of magnitude larger than the cost of floating point operation costs and the gap is increasing \cite{dongarra2014applied}. This requires development of distributed optimization algorithms that can find the right trade-off between the cost of local computations and that of communications.

In this paper, we focus on distributed algorithms for empirical risk minimization problems. The setting is as follows: Given $n$ machines, each machine has access to $m_i$ samples $\{\xi_{i,j}\}_{j=1}^{m_i}$ for $i=1,2,\dots,n$. The samples $\xi_{i,j}$ are random variables supported on a set $\mathcal{P} \subset \reals^d$. Each machine has a loss function that is averaged over the local dataset:
	$$ f_i(\bbx) =  \frac{1}{m_i}\sum_{i=1}^{m_i} \phi(\bbx, \xi_{i,j}) +\frac{\lambda}{2} \| \bbx\|_2^2 $$
where the function $\phi: \reals^p \times \reals^d \to \reals$ is convex in $\bbx$ for each $\xi \in \reals^d$ fixed and $\lambda \geq 0$ is a regularization parameter. The goal is to develop communication-efficient distributed algorithms to minimize the overall empirical loss defined by 
\begin{equation}\label{eq:orig_problem}
   \bbx^* \coloneqq \argmin_{\bbx \in \reals^p}  f(\bbx)
          \coloneqq \argmin_{\bbx \in \reals^p} \frac{1}{n} \sum_{i=1}^n f_i(\bbx).
\end{equation}

The communication model we consider is the \emph{centralized communication} model, also known as the \emph{master/worker} model \cite{xiao2019dscovr}. In this model, the master machine possesses a copy of the global decision variable $\bbx$ which is shared with the worker machines.  Each worker performs local computations based on its local data which is then communicated to the master node to update the decision variable. The way communications are handled can be synchronous or asynchronous, resulting in different type of optimization algorithms and convergence guarantees. The merit of synchronization is that it prevents workers from using obsolete information and, thereby, from submitting a low quality update of parameters to the master. The price to pay, however, is that all the nodes have to wait for the slowest worker, which leads to unnecessary overheads. Asynchronous algorithms do not suffer from this issue, maximizing the efficiency of the workers while minimizing the system overheads. Asynchronous algorithms are particularly preferable over networks with heterogeneous machines with different memory capacities, work overloads, and processing capabilities.

%slower machines work with outdated versions $x^{t-d_i^t}$ of the decision variable where $d_i^t$ is the delay of machine $i$ at time $t$. 
There has been a number of distributed algorithms suggested in the literature to solve the empirical risk minimization problem \eqref{eq:orig_problem} based on primal first-order methods \cite{vanli2016global,gurbuzbalaban2017convergence,dist-adapt-sgd}, their accelerated or  variance-reduced versions \cite{lee17,wai2018sucag,wai2018sucag-journal,asaga,pedregosa2017proxasaga}, lock-free parallel methods \cite{recht2011hogwild,arock}, coordinate descent-based approaches \cite{xiao2019dscovr,takac15dist-sdca,yang2013trading,bianchi2015coordinate}, dual methods \cite{duchi-dual-averaging,yang2013trading}, primal-dual methods \cite{xiao2019dscovr,cocoa,cocoa-plus,bianchi2015coordinate,lazy-iag}, distributed ADMM-like methods \cite{zhang2014asynchronous} as well as quasi-Newton approaches \cite{eisen2017decentralized,wright-dist-qn}, inexact second-order methods \cite{shamir2014dane,reddi2016aide,zhang2015disco,giant2018mahoney,dunner18-trust,gurbuzbalaban2015globally} and general-purpose frameworks for distributed computing environments \cite{cocoa,cocoa-plus} both in the asynchronous and synchronous setting. The efficiency of these algorithms is typically measured by the \emph{communication complexity} which is defined as the equivalent number of vectors in $\mathbb{R}^p$ sent or received across all the machines until the optimization algorithm converges to an $\varepsilon$-neighborhood of the optimum value.  Lower bounds on the communication complexity have been derived in \cite{arjevani2015communication} as well as some linearly convergent algorithms achieving these lower bounds \cite{zhang2015disco,lee17}. However, in an analogy to the lower bounds obtained by \cite{nemirovskii1983problem} for first-order centralized algorithms, the lower bounds for the communication complexity are only effective if the dimension $p$ of the problem is allowed to be larger than the number of iterations. This assumption is perhaps reasonable for very large scale problems where $p$ can be billions, however it is clearly conservative for moderate to large-scale problems where $p$ is not as large. 

\textbf{Contributions:} Most existing state-of-the-art communication-efficient algorithms for strongly convex problems share vectors of size $O(p)$ at every iteration while having linear convergence guarantees. In this work, we propose the first communication-efficient asynchronous optimization algorithm that can achieve superlinear convergence for solving the empirical risk minimization problem under the master/worker communication model. Our algorithm is communication-efficient in the sense that it also shares vectors of size $O(p)$. Our theory supports asynchronous computations subject to both bounded delays and unbounded delays with a bounded time-average. We provide numerical experiments that illustrate our theory and practical performance. The proposed method is based on a distributed asynchronous averaging scheme of decision vectors and gradients in a way to effectively capture the local Hessian information. Our proposed algorithm, Distributed Averaged Quasi-Newton (DAve-QN) % It 
is inspired by the Incremental Quasi-Newton (IQN) method proposed in \cite{mokhtari2018iqn} which is a deterministic incremental algorithm based on the BFGS method. In contrast to the IQN method which is designed for centralized computation, our proposed scheme can be implemented in asynchronous master/worker distributed settings; allowing better scalability properties with  parallelization, while being robust to delays of the workers as an asynchronous algorithm. %Our proposed algorithm, Distributed Averaged Quasi-Newton (DAve-QN) extends the IQN algorithm to the distributed architectures under the master/worker model subject to time delays.  

\textbf{Related work.} Although the setup that we consider in this paper is an asynchronous master/worker distributed setting, it also relates to incremental aggregated algorithms \cite{le2012stochastic,defazio2014saga,defazio2014finito,mairal2015incremental,gurbuzbalaban2017convergence,mokhtari2018surpassing,vanli2018global}, as at each iteration the information corresponding to one of the machines, i.e., functions, is evaluated while the variable is updated by aggregating the most recent information of all the machines. In fact, our method is inspired by an incremental quasi-Newton method proposed in \cite{mokhtari2018iqn} and a delay-tolerant method from~\cite{mishchenko2018delay}.  However, in the IQN method, the update at iteration $t$ is a function of the last $n$ iterates $\{x^{t-1},\dots,x^{t-n} \}$, while in our asynchronous distributed scheme the updates are performed on delayed iterates $\{x^{t-d_1^t-1},\dots,x^{t-d_n^t-n} \}$. This major difference between the updates of these two algorithms requires a challenging different analysis. Further, our algorithm can be considered as an asynchronous distributed variant of traditional quasi-Newton methods that have been heavily studied in the numerical optimization community \cite{goldfarb1970family,broyden1973local,dennis1974characterization,powell1976some}. Also, there have been some works on decentralized variants of quasi-Newton methods for consensus optimization where communications are performed over a fixed arbitrary graph where a master node is impractical or does not exist, this setup is also known as the \emph{multi-agent setting} \cite{nedic2009distributed}. The work in \cite{eisen2017decentralized} introduces a linearly convergent decentralized quasi-Newton method  for decentralized settings. Our setup is different where we have a particular star network topology obeying the master/slave hierarchy. Furthermore, our theoretical results are stronger than those available in the multi-agent setting as we establish a superlinear convergence rate for the proposed method.

\textbf{Outline.} In Section \ref{sec:BFGS_update}, we review the update of the BFGS algorithm that we build on our distributed quasi-Newton algorithm. We formally present our proposed DAve-QN algorithm in Section~\ref{sec:DAve_QN}. We then provide our theoretical convergence results for the proposed DAve-QN method in Section~\ref{sec:convg_analysis}. Numerical results are presented in Section~\ref{sec:exp}. Finally, we give a summary of our results and discuss future work in Section 5.

\vspace{-0.15in}
\section{Algorithm}

\subsection{Preliminaries: The BFGS algorithm}\label{sec:BFGS_update}
The update of the BFGS algorithm for minimizing a convex smooth function $f:\mathbb{R}^p \to \mathbb{R}$ is given by 
\begin{align} \label{eq_descent_update}
\bbx^{t+1} = \bbx^{t} - \eta^t({\bbB^{t+1}})^{-1}\nabla f(\bbx^t),
\end{align}
where $\bbB^{t+1}$ is an estimate of the Hessian $\nabla^2 f(\bbx^t)$ at time $t$ and $\eta^t$ is the stepsize (see e.g. \cite{nocedal2006numerical}). The idea behind the BFGS (and, more generally, behind quasi-Newton) methods is to compute the Hessian approximation $\bbB^{t+1}$ using only first-order information. Like Newton methods, BFGS methods work with stepsize $\eta_t = 1$ when the iterates are close to the optimum. However, at the initial stages of the algorithm, the stepsize is typically determined by a line search for avoiding the method to diverge.

A common rule for the Hessian approximation is to choose it to satisfy the secant condition
	%\begin{equation} 
	$\bbB^{t+1} \bbs^{t+1} = \bby^{t+1} ,
	$%\label{eqn-secant}
	%\end{equation} 
where $ \bbs^{t+1} = \bbx^{t} - \bbx^{t-1},$ and $  \bby^{t+1}  =  \nabla f(\bbx^{t}) - \nabla f(\bbx^{t-1})$ are called the \emph{variable variation} and \emph{gradient variation} vectors, respectively. The Hessian approximation update of BFGS which satisfies the secant condition %in \eqref{eqn-secant} 
can be written as a rank-two update
%%%
\begin{equation}\label{eq-rank-two-update}
 \bbB^{t+1}=\bbB^{t}+\bbU^{t+1}+\bbV^{t+1}, \quad
% \end{equation}
% where
 % \begin{equation}
   \bbU^{t+1} = \frac{\bby^{t+1} ({\bby^{t+1}})^T}{({\bby^{t+1}})^T \bbs^{t+1}}, \ \ \bbV^{t+1} = - \frac{ \bbB^t \bbs^{t+1} (\bbs^{t+1})^T \bbB^{t+1}}{(\bbs^{t+1})^T \bbB^t \bbs^{t+1}}.
\end{equation} 
%%%
Note that both matrices $\bbU^{t}$ and $\bbV^{t}$ are rank-one. Therefore, the update \eqref{eq-rank-two-update} is rank two. Owing to this property, the inverse of the Hessian approximation $\bbB^{t+1}$ can be computed at a low cost of $\mathcal{O}(p^2)$ arithmetic iterations based on the Woodbury-Morrison formula, instead of computing the inverse matrix directly with a complexity of $\mathcal{O}(p^3)$.
 For a strongly convex function $f$ with the global minimum  $\bbx^*$, a classical convergence result for the BFGS method shows that the iterates generated by BFGS are  \emph{superlinearly} convergent \cite{Broyden}, i.e. $ \lim_{t \rightarrow \infty} \frac{ \| \bbx^{t+1} - \bbx^*\|}{ \| \bbx^{t} - \bbx^*\|} = 0
$. There are also limited-memory BFGS (L-BFGS) methods that require less memory ($\mathcal{O}(p)$) at the expense of having a linear (but not superlinear) convergence \cite{nocedal2006numerical}. Our main goal in this paper is to design a BFGS-type method that can solve problem~\eqref{eq:orig_problem} efficiently with superlinear convergence in an asynchronous setting under the master/slave communication model. We introduce our proposed algorithm in the following section.

\vspace{-0.1in}
\subsection{A Distributed Averaged Quasi-Newton Method (Dave-QN)}\label{sec:DAve_QN}

In this section, we introduce a BFGS-type method that can be implemented in a distributed setting (master/slave) without any central coordination between the nodes, i.e., asynchronously. To do so, we consider a setting where $n$ worker nodes (machines) are connected to a master node. Each worker node $i$ has access to a component of the global objective function, i.e., node $i$ has access only to the function $f_i$. The decision variable stored at the master node is denoted by $x^t$ at time $t$. At each moment $t$, $d_i^t$ denotes the delay in communication with the $i$-th worker, i.e., the last exchange with this worker was at time $t - d_i^t$. For convenience, if the last communication was performed exactly at moment $t$, then we set $d_i^t = 0$. In addition, $D_i^t$ denotes the double delay in communication, which relates to the penultimate communication and can be expressed as follows: $D_i^t = d_i^t + d_i^{t - d_i^t - 1} + 1$. Note that the time index $t$ increases if one of the workers performs an update.

Every worker node $i$ has two copies of the decision variable corresponding to the last two communications with the master, i.e. node $i$ possesses $\bbx^{t-d_i^t}$ and 
$\bbz_i^t := \bbx^{t-D_i^t}.$ Since there has been no communication after $t- d_i^t$, we will clearly have \begin{equation} \bbz_i^{t-d_i^t}= \bbz_i^t=\bbx^{t-D_i^t}.
\label{eq-zit-delayed}
\end{equation}%$\bbx$ including an outdated copy of the decision variable that we will denote by $\bbz_i$.
% Note that according to this notation, if node $i_t$ is updated at time $t$, then at time $t+1$ node $i_t$ will possess $x^{i_t}$ as well as 
% $$z_{i_t}^{t+1} = x^{t-d_i^t}$$.
We are interested in designing a distributed version of the BFGS method described in Section \ref{sec:BFGS_update}, where each node at time $t$ has an approximation $\bbB_i^t$ to the local Hessian (Hessian matrix of $f_i$) where 
%. Our goal is to design a method in which each worker $i$ computes its local Hessian approximation matrix
$\bbB_i^t$ is constructed based on the local delayed decision variables $\bbx^{t-d_i^t}$ and $\bbz_i^t$, and therefore the local Hessian approximation will also be outdated satisfying
\begin{equation} \bbB_i^{t}= \bbB_i^{t-d_i^t}.
\label{eq-hessian-delay}
\end{equation}
%The master also possesses an aggregate Hessian 
%This is illustrated in Figure \ref{fig:communication}.
%sends this local information to the master node so that it can update the variable $\bbx^t$ in a way that the sequence of iterates $\bbx^*$ converges superlinearly to the minimum of the global objective function $f$ defined in \eqref{eq:orig_problem}. 
An instance of the setting that we consider in this paper is illustrated in Figure \ref{fig:communication}. At time $t$, one of the workers, say $i_t$, finishes its task and sends a group of vectors and scalars (that we will precise later) to the master node, avoiding communication of any $p\times p$ matrices as it is assumed that this would be prohibitively expensive communication-wise. Then, the master node uses this information to update the decision variable $\bbx^t$ using the new information of node $i_t$ and the old information of the remaining workers. After this process, master sends the updated information to node $i_t$.

%%%%%%%%%%%%%%%%%%%%%%%%%%%%%%%%%%
%%%%% F I G U R E %%%%%%%%%%%%%%%%%%%%%%
\begin{figure}[t]
\begin{center}
    \includegraphics[scale=0.85]{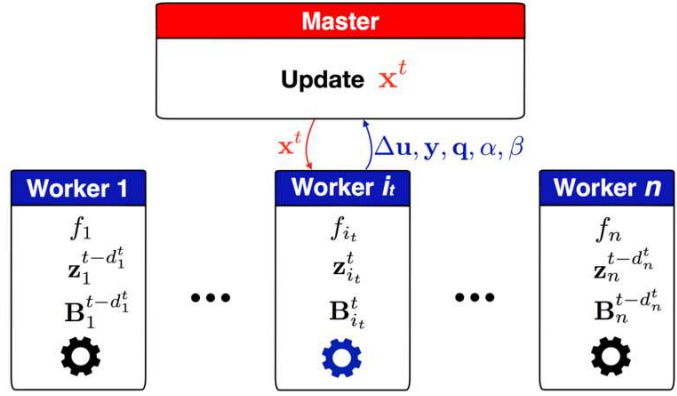}
\end{center}
\caption{Asynchronous communication scheme used by the proposed algorithm.}
\label{fig:communication}
\end{figure}
%%%

% We proceed to formally introducing our proposed Distributed Averaged Quasi-Newton (DAve-QN) method. Each worker will posses two copies $\bbz_i^t$ and $\bbx^{t - d_i^t}$ of the decision variable $\bbx$, where $t-d_i^t$ is the last moment when this worker communicated with the master. Since there has been no communication after $t- d_i^t$, $\bbz_i^t=\bbz_i^{t-d_i^t}$. The meaning of $\bbz_i^t$ is the value of $\bbx$ used to produce the update sent by $i$-th worker at $t-d_i^t$. Therefore, it is also equal to $\bbx^{t-D_i^t}$, where $t-D_i^t$ is the time when the worker received the master variable that was before $\bbx^{t-d_i^t}$. 

We define the \emph{aggregate Hessian approximation} as
\begin{equation}\label{eq-hessian-delayed} \bbB^t :=\sum_{i=1}^n \bbB_i^{t}= \sum_{i=1}^n \bbB_i^{t-d_i^t} \end{equation} where we used \eqref{eq-hessian-delay}. In addition, we introduce  \begin{equation}\bbu^t := \sum_{i=1}^n \bbB_i^t \bbz_i^t =  \sum_{i=1}^n \bbB_i^{t-d_i^t} \bbz_i^{t-d_i^t}, \quad \bbg^t := \sum_{i=1}^n \nabla f_i (\bbz_i^t)=\sum_{i=1}^n \nabla f_i (\bbz_i^{t-d_i^t})\label{eq-u-g}
\end{equation} as the \emph{aggregate Hessian-variable product} and  \emph{aggregate gradient} respectively where we made use of the identities \eqref{eq-zit-delayed}--\eqref{eq-hessian-delay}. All these vectors and matrices are only available at the master node since it requires access to the information of all the workers. 

Given that at step $t+1$ only a single index $i_t$ is updated, using the identities \eqref{eq-zit-delayed}--\eqref{eq-u-g}, it follows that the master has the update rules
\begin{align}
\bbB^{t+1} &= \bbB^{t} +  \left( \bbB_{i_t}^{t+1} - \bbB_{i_t}^{t} \right) =  \bbB^{t} +  \left( \bbB_{i_t}^{t} - \bbB_{i_t}^{t-d_i^t} \right), \label{eq:quick_updates_1}\\ 
\bbu^{t+1} &= \bbu^{t} +\left( \bbB_{i_t}^{t+1} \bbz_{i_t}^{t+1} - \bbB_{i_t}^{t}\bbz_{i_t}^{t} \right) = \bbu^{t} +\left( \bbB_{i_t}^{t+1} \bbx^{t-d_{i_t}^t} - \bbB_{i_t}^{t-d_{i_t}^t}\bbx^{t-D_{i_t}^t} \right),\label{eq:quick_updates_2}
 \\ \bbg^{t+1} &= \bbg^{t} +  \left(\nabla f_{i_t} (\bbz_{i_t}^{t+1}) - \nabla f_{i_t}(\bbz_{i_t}^{t}) \right) = \bbg^{t} +  \left(\nabla f_{i_t} (\bbx^{t-d_{i_t}^t}) - \nabla f_{i_t}(\bbx^{t-D_{i_t}^t}) \right). \label{eq:quick_updates_3}
\end{align}
We observe that, only $\bbB_{i_t}^{t+1}$ and $\nabla f_{i_t} (\bbz_{i_t}^{t+1})=\nabla f_{i_t} (\bbx^{t-d_{i_t}^t})$ are required to be computed at step $t + 1$. The former is obtained by the standard BFGS rule applied to $f_i$ carried out by the worker $i_t$:
\begin{align}\label{worker_update}
	\bbB_{i_t}^{t+1} = \bbB_{i_t}^{t} + \frac{\bby_{i_t}^{t+1}(\bby_{i_t}^{t+1})^{\top}}{\alpha^{t+1}} - \frac{\bbq_{i_t}^{t+1}(\bbq_{i_t}^{t+1})^{\top}}{\beta^{t+1}}
\end{align}
with \begin{equation}\label{eq-alpha} \alpha^{t+1} \coloneqq (\bby_{i_t}^{t+1})^{\top}\bbs_{i_t}^{t+1} \quad \bby_{i_t}^{t+1}\coloneqq \bbz_{i_t}^{t+1} - \bbz_{i_t}^{t} = \bbx^{t-d_{i_t}^t} - \bbx^{t-D_{i_t}^t},\end{equation} 
\begin{equation}\label{eq-beta}\bbq_{i_t}^{t+1}\coloneqq \bbB_{i_t}^{t}\bbs_{i_t}^{t+1}, \quad\beta^{t+1}\coloneqq (\bbs_{i_t}^{t+1})^{\top} \bbB_{i_t}^{t}\bbs_{i_t}^{t+1} =  (\bbs_{i_t}^{t+1})^{\top} \bbq_{i_t}^{t+1}.\end{equation}

Then, the master computes the new iterate as
%\begin{align}\label{eq:incr_bfgs_update_short}
$\bbx^{t+1} =(\bbB^{t+1})^{-1} \left( \bbu^{t+1}- \bbg^{t+1} \right)$
%\end{align}
and sends it to worker $i_t$. For the rest of the workers, we update the time counter without changing the variables, so $\bbz_i^{t+1} = \bbz_i^{t}$ and $\bbB_i^{t+1} = \bbB_i^{t}$ for $i\neq i_t$. Although, updating the inverse $\bbB^{t})^{-1}$ %\eqref{eq:incr_bfgs_update_short} 
may seem costly first glance, in fact it can be computed efficiently in ${\cal O}(p^2)$ iterations, similar to standard implementations of the BFGS methods. More specifically, if we introduce a new matrix
\begin{align}\label{eq:incr_bfgs_update_inverse_aux}
\bbU^{t+1} \coloneqq (\bbB^{t})^{-1} -\frac{(\bbB^{t})^{-1} \bby_{i_t}^{t+1}{(\bby_{i_t}^{t+1})^{\top}}(\bbB^{t})^{-1}}{(\bby_{i_t}^{t+1})^{\top} \bbs_{i_t}^{t+1} + ({\bby_{i_t}^{t})^{\top}}(\bbB^{t})^{-1} \bby_{i_t}^{t+1}},
\end{align}
then, by the Sherman-Morrison-Woodbury formula, we have the identity
\begin{align}\label{eq:incr_bfgs_update_inverse}
&(\bbB^{t+1})^{-1}=\bbU^{t+1} +\frac{\ \!\bbU^{t+1} (\bbB_{i_t}^{t-d_{i_t}^t} \bbs_{i_t}^{t+1})(\bbB_{i_t}^{t-d_{i_t}^t} \bbs_{i_t}^{t+1})^T\bbU^{t+1}}{(\bbs_{i_t}^{t+1})^T\bbB_{i_t}^{t-d_{i_t}^t} \bbs_{i_t}^{t+1}\!-\!(\bbB_{i_t}^{t-d_i^t} \bbs_{i_t}^{t+1})^T\bbU^{t+1} (\bbB_{i_t}^{t-d_{i_t}^t} \bbs_{i_t}^{t+1})},
\end{align}
Therefore, if we already have $(\bbB^{t})^{-1}$, it suffices to have only matrix vector products. %it will require us only to multiply and add vectors and matrices. 
If we denote $\bbv^{t+1} = (\bbB^{t})^{-1}\bby_{i_t}^{t+1}$ and $\bbw^{t+1} \coloneqq \bbU^{t+1} \bbq_{i_t}^{t+1}$, then these equations can be simplified as
\begin{align}\label{eff_perform}
	\bbU^{t+1} &= (\bbB^{t})^{-1} - \frac{\bbv^{t+1}(\bbv^{t+1})^{\top}}{\alpha^{t+1} + (\bbv^{t+1})^{\top}\bby_{i_t}^{t+1}}, \quad 
	\bbv^{t+1} = (\bbB^{t})^{-1}\bby_{i_t}^{t+1},\\
	(\bbB^{t+1})^{-1} &= \bbU^{t+1} + \frac{\bbw^{t+1}(\bbw^{t+1})^{\top}}{\beta^{t+1} - (\bbq^{t+1})^{\top}\bbw^{t+1}}, \quad \bbw^{t+1} \coloneqq \bbU^{t+1} \bbq_{i_t}^{t+1}, \label{eff_perform-2}
\end{align}
\vspace{-0.1in}
\tcbset{width=1.0\columnwidth,before=,after=, colframe=black,colback=white, fonttitle=\bfseries, coltitle=white, colbacktitle=red!80!yellow, boxrule=0.2mm}
\begin{algorithm*}[ht]
\caption{DAve-QN (implementation)}
\label{alg:DQN_1}
\centering
\begin{multicols}{2}
\begin{tcolorbox}[title=Master:]
\begin{small}
\textbf{Initialize} $ \bbx$, $\bbB_i$, $\bbg = \sum_{i=1}^n\nabla f_i(\bbx)$, $\bbB^{-1} = (\sum_{i=1}^n \bbB_i)^{-1}, \bbu = \sum_{i=1}^n \bbB_i \bbx$, \\
%Send $\overline x$ to each machine\\
\For{$t= 1$ \KwTo $T-1$}{
    \textbf{If} a worker sends an update:\\
    {\color{blue!70!black} Receive $\Delta \bbu$, $\bby$, $\bbq$, $\alpha$, $\beta$ from it}\\
    $\bbu= \bbu + \Delta \bbu$,
    $\bbg = \bbg + \bby$,
%     $\bbB_i^{t} = \bbB_i^{t-1} + \frac{\bby \bby^T}{\alpha} - \frac{\bbz \bbz^T}{\beta}$\\
     $\bbv = (\bbB)^{-1} \bby$\\
    $\bbU = (\bbB)^{-1} - \frac{\bbv \bbv^\top}{\alpha + \bbv^\top\bby}$\\
    $\bbw = \bbU \bbq$,
    $(\bbB)^{-1} = \bbU + \frac{\bbw \bbw^\top}{\beta - \bbq^\top\bbw}$\\
    $\bbx = (\bbB)^{-1} (\bbu - \bbg)$\\
    {\color{red!80!yellow} Send $\bbx$ to the worker in return}
}
Interrupt all workers\\
\textbf{Output} $ \bbx^T $
\end{small}
\end{tcolorbox}

\columnbreak
\tcbset{width=1.03\columnwidth,before=, colframe=black!50!black, colbacktitle=blue!70!black}
\begin{tcolorbox}[title=Worker $i$:]
\begin{small}
\textbf{Initialize} $\bbx_i = \bbx$, $\bbB_i$\\
\While{not interrupted by master}{
    {\color{red!80!yellow} Receive $\bbx$}\\
%     $\bbB_i^t = \bbB_i^{t - D_i^t}$\\
    $\bbs_i  = \bbx - \bbz_i$\\	
    $\bby_i = \nabla f_i(\bbx) - \nabla f_i(\bbz_i)$\\
    $\bbq_i = \bbB_i \bbs_i$\\
    $\alpha = \bby_i^{\top} \bbs_i$\\
    $\beta = \bbs_i^\top \bbB_i^{t} \bbs_i$\\
    $\bbu = \bbB_i \bbz_i$\\
    $\bbB_i = \bbB_i + \frac{\bby_i \bby_i^{\top}}{\alpha} - \frac{\bbq_i \bbq_i^{\top}}{\beta}$\\
    $\Delta \bbu = \bbB_i\bbx - \bbu$\\
    $\bbz_i = \bbx$\\
    {\color{blue!70!black} Send $\Delta \bbu, \bby_i, \bbq_i, \alpha, \beta$ to the master}
}
\end{small}
\end{tcolorbox}
\end{multicols}
\end{algorithm*}
where $\alpha^{t+1},\beta^{t+1},\bbq_{i_t}^{t+1}$ and $\bby_{i_t}^{t+1}$ are defined by \eqref{eq-alpha}--\eqref{eq-beta}. 

The steps of the DAve-QN at the master node and the workers are summarized in Algorithm~\ref{alg:DQN_1}. Note that steps at worker $i$ is devoted to performing the update in \eqref{worker_update}. Using the computed matrix $\bbB_i$, node $i$ evaluates the vector $\Delta \bbu$. Then, it sends the vectors $\Delta \bbu$, $\bby_i$, and $\bbq_i$ as well as the scalars $\alpha$ and $\beta$ to the master node. The master node uses the variation vectors $\Delta \bbu$ and $\bby$ to update $\bbu $ and $\bbg$. Then, it performs the update %in \eqref{eq:incr_bfgs_update_short} 
$\bbx^{t+1} =(\bbB^{t+1})^{-1} \left( \bbu^{t+1}- \bbg^{t+1} \right)$
by following the efficient procedure presented in \eqref{eff_perform}--\eqref{eff_perform-2}. A more detailed version of Algorithm 1 with exact indices is presented in the supplementary material.

We define \emph{epochs} $\{T_m\}_m$ by setting $T_1=0$ and the following recursion:
\begin{align*} 
T_{m+1}&=\min\{t:\text{ each machine made at least 2 updates on the interval }[T_m, t]\}\\
&= \min\{t: t-D_i^t \geq T_m \text{ for all } i=1,..,M  \}.
\end{align*}
%%%

%%%%%%%%%%%%%%%%%%%%%%%%%%%%%%%%%%
%%%%%%%%%%%%%%%%%%%%%%%%%%%%%%%%%%
%%%%%    L   E   M   M   A       %%%%%%%%%%%%%%%%%%
%%%%%%%%%%%%%%%%%%%%%%%%%%%%%%%%%%
%%%%%%%%%%%%%%%%%%%%%%%%%%%%%%%%%%
The proof of the following simple lemma is provided in the supplementary material.
\begin{lemma}\label{lemm_update}
	Algorithm \ref{alg:DQN} iterates satisfy %possesses the following property:
    %\begin{align*}
     $   \bbx^{t} = \left(\frac{1}{n}\sum_{i=1}^n \bbB_i^t \right)^{-1} \left(\frac{1}{n}\sum_{i=1}^n\bbB_i^t \bbz_i^t - \frac{1}{n}\sum_{i=1}^n \nabla f_i(\bbz_i^t) \right)$.
    %\end{align*}
\end{lemma}

The result in Lemma \ref{lemm_update} shows that explicit relationship between the updated variable $\bbx^t$ based on the proposed DAve-QN and the local information at the workers. We will use this update to analyze DAve-QN.

%%%%%%%%%%%%%%%%%%%%%%%%%%%%%%%%%%
%%%%%%%%%%%%%%%%%%%%%%%%%%%%%%%%%%
%%%%%    P R O P O S I T I O N       %%%%%%%%%%%%%%%%%%
%%%%%%%%%%%%%%%%%%%%%%%%%%%%%%%%%%
%%%%%%%%%%%%%%%%%%%%%%%%%%%%%%%%%%
\begin{proposition}[Epochs' properties]\label{pr:delays_recurrence}
	The following relations between epochs and delays hold:
    \begin{itemize}
    	\item For any $t\in [T_{m+1}, T_{m+2})$ and any $i=1,2,\dots,n$ one has $t - D_i^t\in [T_m, t)$.
        \item If delays are uniformly bounded, i.e.\ there exists a constant $d$ such that $d_i^t\le d$ for all $i$ and $t$, then for all $m$ we have $T_{m + 1} - T_m \le D\coloneqq 2d+1$ and $T_m\le Dm$.
        \item If we define average delays as $\overline {d^t}\coloneqq \tfrac{1}{n}\sum_{i=1}^n d_i^t$, then $\overline{d^t}\ge (n - 1) / 2$. Moreover, assuming that $\overline{d^t} \le (n- 1)/2 + \overline d$ for all $t$, we get $T_m\le 4n (\overline d + 1)m$.
    \end{itemize}
\end{proposition}
Clearly, without visiting every function we can not converge to $\bbx^*$. Therefore, it is more convenient to  measure performance in terms of number of passed epochs, which can be considered as our alternative counter for time. Proposition \ref{pr:delays_recurrence} explains how one can get back to the iterations time counter assuming that delays are bounded uniformly or on average. However, uniform upper bounds are rather pessimistic which motivates the convergence in epochs that we consider. %we shall discuss is much more general.
%!TEX root = ex_article_v2.tex
\vspace{-0.15in}
\section{Convergence Analysis}\label{sec:convg_analysis}
\vspace{-0.1in}
In this section, we study the convergence properties of the proposed distributed asynchronous quasi-Newton method. To do so, we first assume that the following conditions are satisfied. 

\begin{assumption} \label{ass:gradient_assumption}
The component functions $f_i$ are $L$-smooth and $\mu$-strongly convex, i.e., there exist positive constants $0 < \mu \leq L$ such that, for all $i$ and $\bbx, \hbx\in \reals^p$
\begin{equation}\label{eq:gradient_assumption}
\mu \| \bbx - \hbx \|^2 \leq(\nabla f_i (\bbx) - \nabla f_i(\hbx) )^T(\bbx - \hbx) \leq L \| \bbx - \hbx \|^2.
\end{equation}
\end{assumption}
%
%The second assumption is that the component Hessians are also Lipschitz continuous. 
%%%%%%%%%%
\begin{assumption} \label{ass:hessian_assumption}
The Hessians $ \nabla^2 f_i$ are Lipschitz contunuous, i.e., there exists a positive constant $\tilde{L}$ such that, for all $i$ and $\bbx, \hbx\in \reals^p$, we can write
%
%\begin{equation}\label{eq:hessian_assumption}
 $\| \nabla^2 f_i (\bbx) - \nabla^2 f_i(\hbx) \| \leq \tilde{L} \| \bbx - \hbx \|$.
%\end{equation}
%
\end{assumption}
\vspace{-0.05in}
It is well-known and widely used in the literature on Newton's and quasi-Newton methods \cite{nesterov2013introductory,Broyden,Powell,Dennis} that if the function $f_i$ has Lipschitz continuous Hessian $x\mapsto\nabla^2 f_i(x)$ with parameter $\tilde{L}$ then
%%%
\begin{equation}\label{eq:result_of_lip_hessian}
\left\| \nabla^2 f_i(\tbx) (\bbx-\hbx) -(\nabla f_i(\bbx)-\nabla f_i(\hbx))\right\|
\leq {\tilde{L}} \|\bbx-\hbx\| \max\left\{\|\bbx-\tbx\|,\|\hbx-\tbx\|\right\},
\end{equation}
for any arbitrary $\bbx,\tbx,\hbx\in \reals^p$. See, for instance, Lemma~3.1 in \cite{Broyden}.

%%%%%%%%%%%%%%%%%%%%%%%%%%%%%%%%%%
%%%%%%%%%%%%%%%%%%%%%%%%%%%%%%%%%%
%%%%%    L   E   M   M   A       %%%%%%%%%%%%%%%%%%
%%%%%%%%%%%%%%%%%%%%%%%%%%%%%%%%%%
%%%%%%%%%%%%%%%%%%%%%%%%%%%%%%%%%%

\begin{lemma} \label{lem:M_2}
Consider the Dave-QN algorithm summarized in Algorithm \ref{alg:DQN}.
For any $i$, define the residual sequence for function $f_i$ as $\sigma_i^t:= \max\{\|\bbz_i^{t}-\bbx^*\|,\|\bbz_i^{t-D_i^t}-\bbx^*\|\}$ and set $\bbM_i=\nabla^2 f_i(\bbx^*)^{-1/2}$. If Assumptions \ref{ass:gradient_assumption} and \ref{ass:hessian_assumption} hold and the condition $  \sigma_i^t<\mu/(3\tilde{L})$ is satisfied then a Hessian approximation matrix $\bbB_i^t$ and its last updated version $\bbB_i^{t - D_i^t} $ satisfy
\begin{align}\label{eq:lemma_M_2_claim}
\begin{small}
\left\|\bbB_i^{t}\!-\! \nabla^2 f_i(\bbx^*)\right\|_{\bbM_i}
\leq \!\left[ \left[1\!-\! \alpha {\theta_i^{t - D_i^t}}^2\right]^{\frac{1}{2}} 
	  \!\!\!+\alpha_3 \sigma_i^{t - D_i^t} \right] \!\! \left\|\bbB_i^{t - D_i^t}- \nabla^2 f_i(\bbx^*)\right\|_{\bbM_i} \!\!+\! \alpha_4 \sigma_i^{t - D_i^t},
\end{small}	  
\end{align}
where $\alpha,\alpha_3 $, and $\alpha_4$ are some positive constants and 
%\begin{align}\label{def_theta}
$\theta_i^{t - D_i^t} =\frac{\|\bbM_i(\bbB_i^{t - D_i^t}- \nabla^2 f_i(\bbx^*))\bbs_i^{t - D_i^t}\|}{\|\bbB_i^{t - D_i^t}- \nabla^2 f_i(\bbx^*)\|_{\bbM_i}\|\bbM_i^{-1}\bbs_i^{t - D_i^t}\|}$
%\end{align}
with the convention that $\theta_i^{t - D_i^t}=0$ in the special case $\bbB_i^{t - D_i^t} =  \nabla^2 f_i(\bbx^*)$.
%%%
\end{lemma}

%\begin{proof}
%The proof is given in the supplementarity material.
%\end{proof}
%%%

Lemma~\ref{lem:M_2} shows that, if we neglect the additive term $ \alpha_4 \sigma_i^{t - D_i^t}$ in \eqref{eq:lemma_M_2_claim}, the difference between the Hessian approximation matrix $\bbB_i^{t}$ for the function $f_i$ and its corresponding Hessian at the optimal point $ \nabla^2 f_i(\bbx^*)$ decreases by following the update of Algorithm~\ref{alg:DQN}. To formalize this claim and show that the additive term is negligible, we prove in the following lemma that the sequence of errors $\|\bbx^t-\bbx^*\|$ converges to zero R-linearly which also implies linear convergence of the sequence $\sigma_i^t$.

%%%%%%%%%%%%%%%%%%%%%%%%%%%%%%%%%%
%%%%%%%%%%%%%%%%%%%%%%%%%%%%%%%%%%
%%%%%    L   E   M   M   A       %%%%%%%%%%%%%%%%%%
%%%%%%%%%%%%%%%%%%%%%%%%%%%%%%%%%%
%%%%%%%%%%%%%%%%%%%%%%%%%%%%%%%%%%
\begin{lemma}\label{lemma:lin_convergence}
Consider the Dave-QN method outlined in Algorithm \ref{alg:DQN}. Further assume that the conditions in Assumptions~\ref{ass:gradient_assumption} and \ref{ass:hessian_assumption} are satisfied. Then, for any $r\in(0,1)$ there exist positive constants $\eps(r)$ and $\delta(r)$ such that if $\|\bbx^0-\bbx^*\|< \eps(r)$ and $\|\bbB_i^0-\nabla^2 f_i(\bbx^*)\|_\bbM< \delta(r)$ for $\bbM=\nabla^2 f_i(\bbx^*)^{-1/2}$ and $i=1,2,\dots,n$, the sequence of iterates generated by DAve-QN satisfy
%%%%
%\begin{equation}\label{eq:linear_conv_for_arbitrary_r}
$\|\bbx^{t}-\bbx^*\|\leq r^{m} \|\bbx^0-\bbx^*\|$
%\end{equation}
for all $t\in [T_{m}, T_{m+1})$.
\end{lemma}
%\begin{proof}
%The proof is given in the supplementary material.
%\end{proof}

The result in Lemma \ref{lemma:lin_convergence} shows that the error for the sequence of iterates generated by the Dave-QN method converge to zero at least linearly in a neighborhood of the optimal solution. Using this result, in the following theorem we prove our main result, which shows a specific form of superlinear convergence.

%%%%%%%%%%%%%%%%%%%%%%%%%%%%%%%%%%
%%%%%%%%%%%%%%%%%%%%%%%%%%%%%%%%%%
%%%%%    T  H  E  O  R  E  M       %%%%%%%%%%%%%%%%%%
%%%%%%%%%%%%%%%%%%%%%%%%%%%%%%%%%%
%%%%%%%%%%%%%%%%%%%%%%%%%%%%%%%%%%
\begin{theorem}\label{thm:sup_linear_thm}
Consider the proposed method outlined in Algorithm \ref{alg:DQN}. Suppose that Assumptions~\ref{ass:gradient_assumption} and \ref{ass:hessian_assumption} hold. Further, assume that the required conditions for the results in Lemma~\ref{lem:M_2} and Lemma~\ref{lemma:lin_convergence} are satisfied. Then, the sequence of residuals $ \| \bbx^t - \bbx^*\|$ satisfies 
%
%\begin{equation}\label{eq:final_claim}
$\lim_{t \rightarrow \infty} \frac{ \max_{t\in [T_{m+1}, T_{m+2})}\| \bbx^{t} - \bbx^*\|}{\max_{t \in [T_m, T_{m+1})} \|\bbx^t - \bbx^*\|} = 0$.
%\end{equation}
%
\end{theorem}
%\begin{proof}
%The proof is given in the supplementary material. 
%\end{proof}
\vspace{-0.05in}
The result in Theorem \ref{thm:sup_linear_thm} shows that the maximum residual in an  epoch divided by the  the maximum residual  for the previous epoch converges to zero. This observation shows that there exists a subsequence of residuals  $\| \bbx^{t} - \bbx^*\|$ that converges to zero superlinearly.

%!TEX root = ex_article_v2.tex
\vspace{-0.05in}
\section{Experiments}\label{sec:exp}
\vspace{-0.05in}
We conduct our experiments on five datasets (\verb+epsilon+, \verb+SUSY+, \verb+covtype+, \verb+mnist8m+, \verb+cifar10+) from the LIBSVM library \cite{chang2011libsvm}.\footnote{We use all the datasets without any pre-processing except for the smaller-scale covtype dataset, which we enlarged 5 times for bigger scale experiments using the approach in \cite{giant2018mahoney}. % based on the Python code from  \url{https://github.com/wangshusen/SparkGiant/blob/master/data/RandNoise.py}.
} For the first three datasets, the objective considered is a binary logistic regression problem 
$f(x) = \frac{1}{n}\sum_{i=1}^n \log(1+\exp(-b_i a_i^T x) +\frac{\lambda}{2}\|x\|^2$ where $a_i \in \mathbb{R}^p$ are the feature vectors and
%$f(x) = \frac{1}{n} \sum_{i=1}^n \log(1 + \exp(−b_i a_i^T x) + \frac{\lambda}{2} \| x\|^2$ where $a_i \in \mathbb{R}^d$ are the feature vectors and 
$b_i \in \{-1, +1\}$ are the labels. The other two datasets are about multi-class classification instead of binary classification. %We set the regularization parameter to $\lambda = 1/n$ for all the datasets. 
%\begin{align*}
% 	\min\frac{1}{2}\left\|A x - b\right\|_2^2,
% \end{align*}
%where the matrix $A$ contains the features and $b$ is the output vector. 
For comparison, we used two other algorithms designed for distributed optimization:
\begin{itemize}
	%\item Proximal Incremental Average Gradient (\textbf{PIAG}). A first-order method, which is similar to the SAG/SAGA method, but does not assume uniform sampling of the functions.
	\item Distributed Average Repeated Proximal Gradient (\textbf{DAve-RPG}) \cite{mishchenko2018delay}. It is a recently proposed competitive state-of-the-art asynchronous method for first-order distributed optimization, numerically demonstrated to outperform incremental aggregated gradient methods \cite{gurbuzbalaban2017convergence,vanli2018global} and synchronous proximal gradient methods in \cite{,mishchenko2018delay}.
	\item Distributed Approximate Newton (\textbf{DANE}) \cite{shamir2014dane}. This is a well-known Newton-like method that does not require a parameter server node, but performs reduce operations at every step. 
%	\item Distributed SVRG (\textbf{D-SVRG}) \cite{lee17}:
	%It is a close relative of our method as it also estimates second-order information about the function using first-order operations only.
\end{itemize}
	In the experiments we did not implement algorithms that require shared memory (such as ASAGA \cite{asaga} or Hogwild! \cite{recht2011hogwild}) because in our setting of master/worker communication model, the memory is not shared.   Since the focus of this paper is mainly on asynchronous algorithms where the communication delays is the main bottleneck, for fairness reasons, we are also not comparing our method with some synchronous algorithms such as DISCO \cite{zhang2015disco} that would not support asynchronous computations. Our code is publicly available at \url{https://github.com/DAve-QN/source}.
	%they are not distributed. In contrast, these algorithms require at each iteration new data to be sampled meaning that the memory is shared. This is not our setting as we assume that each computational node knows only its own function. 
%	Nevertheless, we use the distributed counterpart of the SAG method, which is PIAG.

	 The experiments are conducted on XSEDE Comet CPUs (Intel Xeon E5-2680v3 2.5 GHz). For DAVE-QN and DAVE-RPG we build a cluster of 17 processes in which 16 of the processes are workers and one is the master. % on one machine with
   % multiple processes where each process simulate a nodes. 
   The DANE method does not require a master so we use 16 workers for its experiments. We split the data randomly among the processes so that each has the same amount of samples. In our experiments, Intel MKL 11.1.2 and MVAPICH2/2.1 are used for the BLAS (sparse/dense) operations and we use MPI programming compiled with mpicc 14.0.2. Each experiment is repeated thirty times and the average is reported. 
    
    % \textit{parameters}: We use the
	%In our Python implementation, all communications were handled using basic methods of MPI4Py [dalcinmpi4py]. 

	For the methods' parameters the best options provided by the method authors are used. %For PIAG we used stepsize $\tfrac{1}{3Ln}$ as suggested in \citep{vanli2018global}, and 
	For DAve-RPG the stepsize $\tfrac{1}{L}$ is used where $L$ is found by a standard backtracking line search similar to \cite{schmidt2015non}. DANE has two parameters, $\eta$ and $\mu$. As recommended by the authors, we use $\eta = 1$ and $\mu=3\lambda$. We tuned $\lambda$ to the dataset, choosing $\lambda=1$ for the \verb+mnist8m+ and \verb+cifar10+ datasets, $\lambda=0.001$ for the \verb+epsilon+  and \verb+SUSY+ and $\lambda=0.1$ for the \verb+covtype+.
	%proportional to strong convexity constant. Only the former is supported by the theory, but as the latter provides better results, we included both. 
	Since DANE requires a local optimization problem to be solved, we use SVRG \cite{johnson2013accelerating} as its local solver where its parameters are selected based on the experiments in \cite{shamir2014dane}.
	
	Our results are summarized in Figure \ref{fig:suboptimality} where we report the expected suboptimality versus time in a logarithmic y-axis. For linearly convergent algorithms, the slope of this plot determines the convergence rate. DANE method is the slowest on these datasets, but it does not need a master, therefore it can apply to multi-agent applications \cite{nedic2009distributed} where master nodes are often not available. We observe that Dave-QN performs significantly better on all the datasets except \verb+cifar10+, illustrating the superlinear convergence behavior provided by our theory compared to other methods. For the \verb+cifar10+ dataset, $p$ is the largest. Although Dave-QN starts faster than Dave-RPG , Rave-RPG has a cheaper iteration complexity ($O(p)$ compared to $O(p^2)$ of Dave-QN) and becomes eventually faster. %We will also investigate to use sparse matrix vector computations in our implementation to improve the performance of Dave-QN further when $p$ can be very large for sparse datasets.
\vspace{-0.1in}
\begin{figure} [H]
\centering
\begin{tabular}{cccc}
\includegraphics[width=0.3\textwidth]{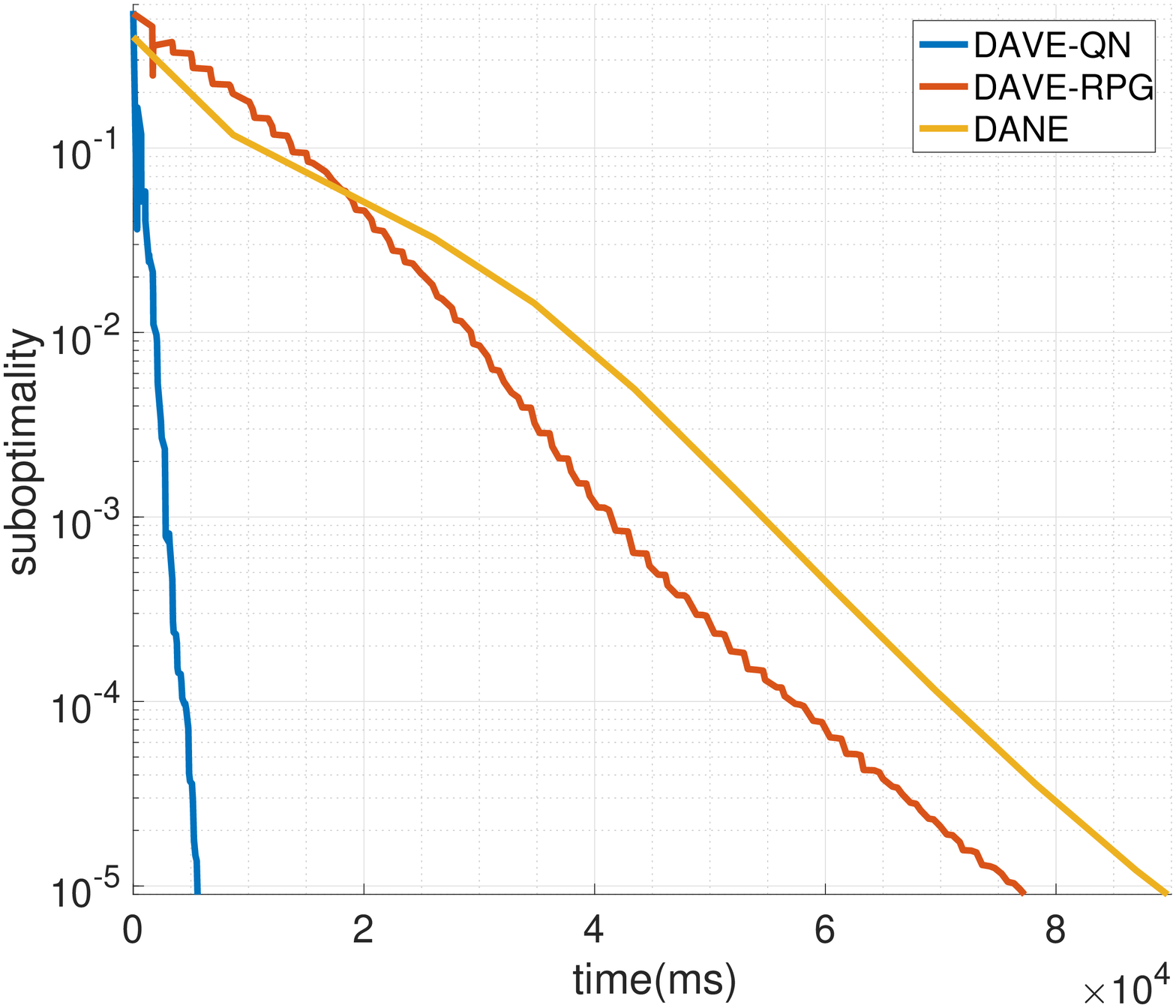} &
\includegraphics[width=0.3\textwidth]{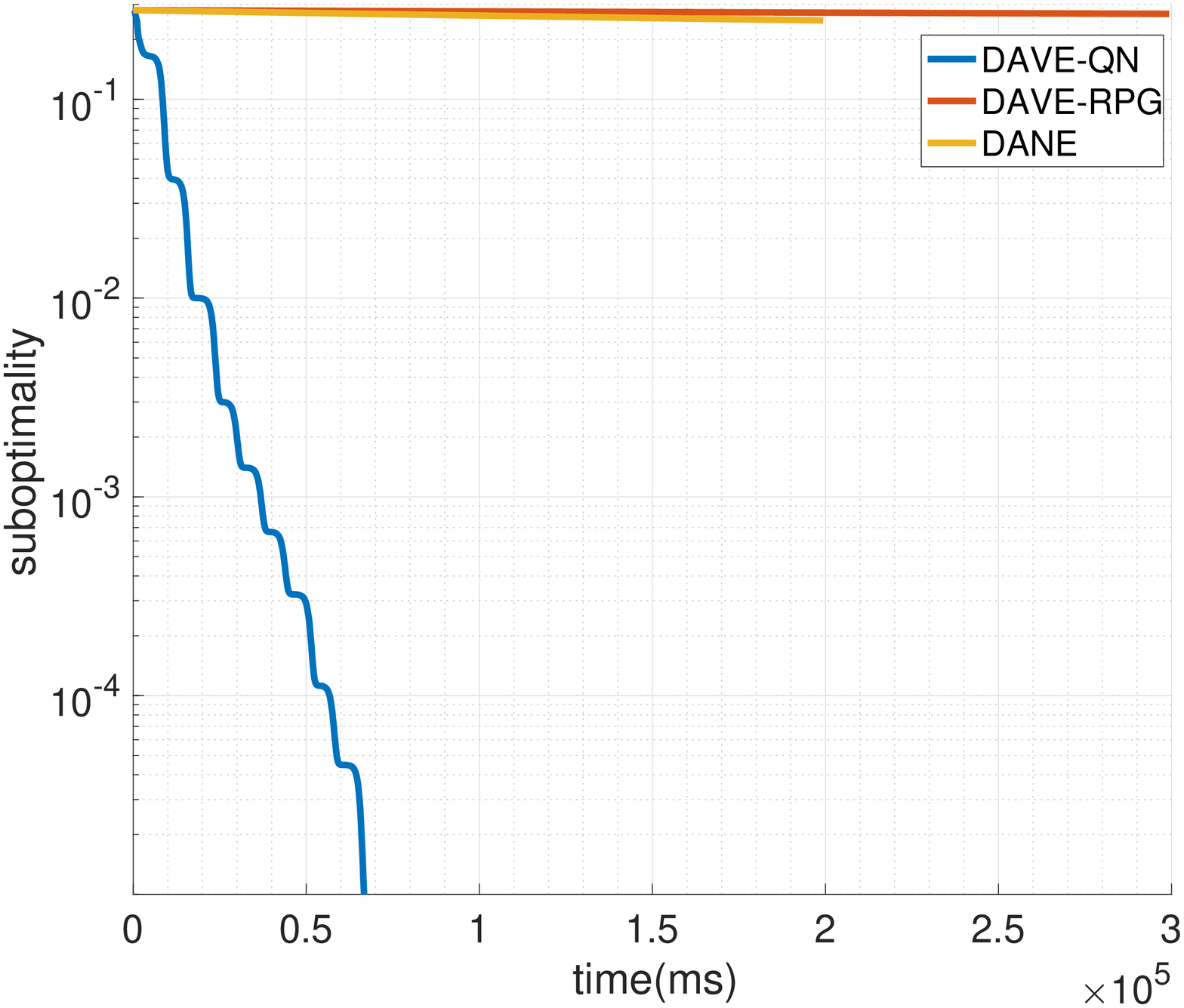} &
\includegraphics[width=0.3\textwidth]{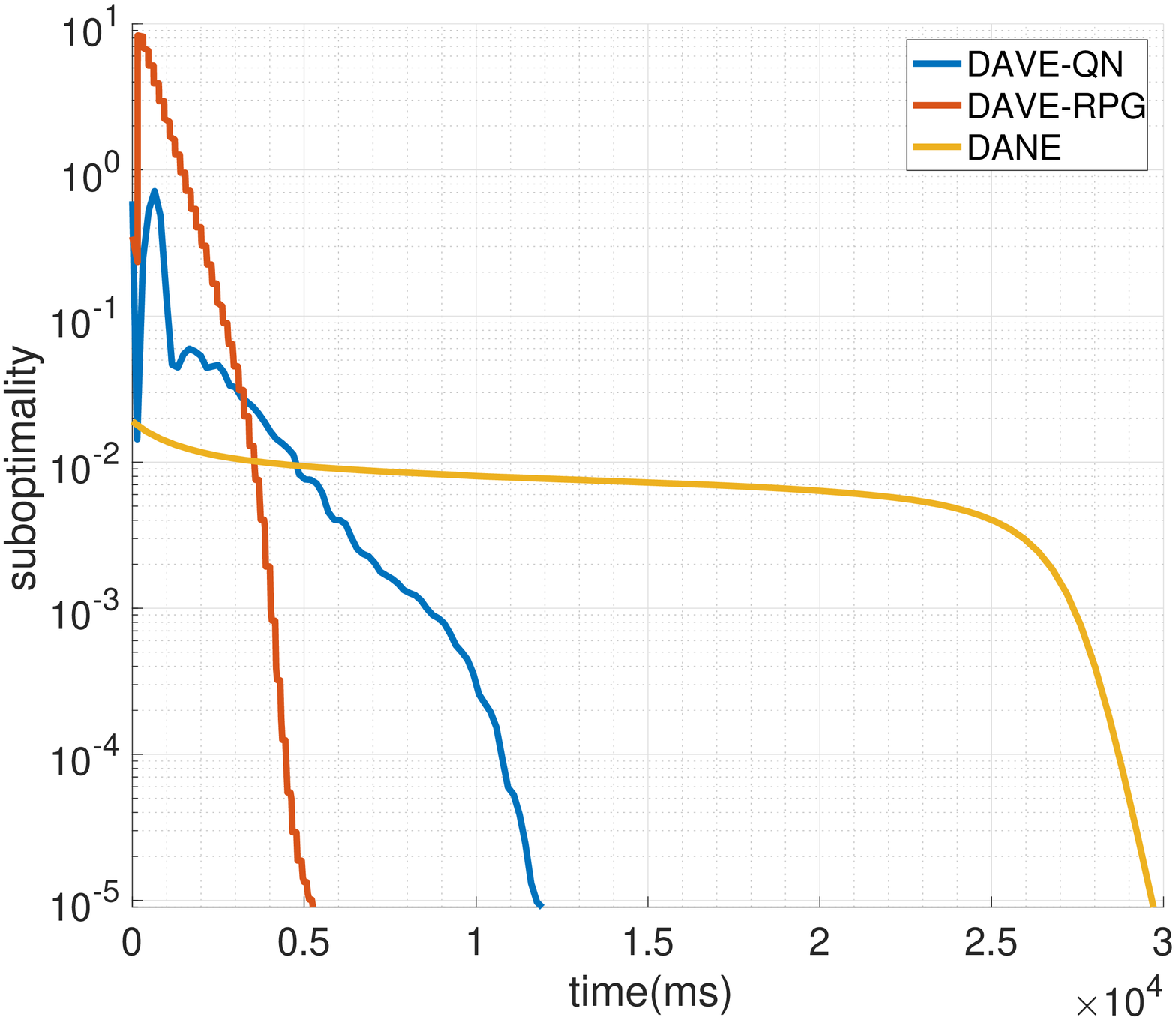} \\
\textbf{mnist8m $(p=784,n=8.1M)$}  & \textbf{epsilon $(p=2000,n=500K)$} & \textbf{cifar10 $(p=3072,n=60K)$}  \\[6pt]
\end{tabular}
\begin{tabular}{cccc}
\includegraphics[width=0.33\textwidth]{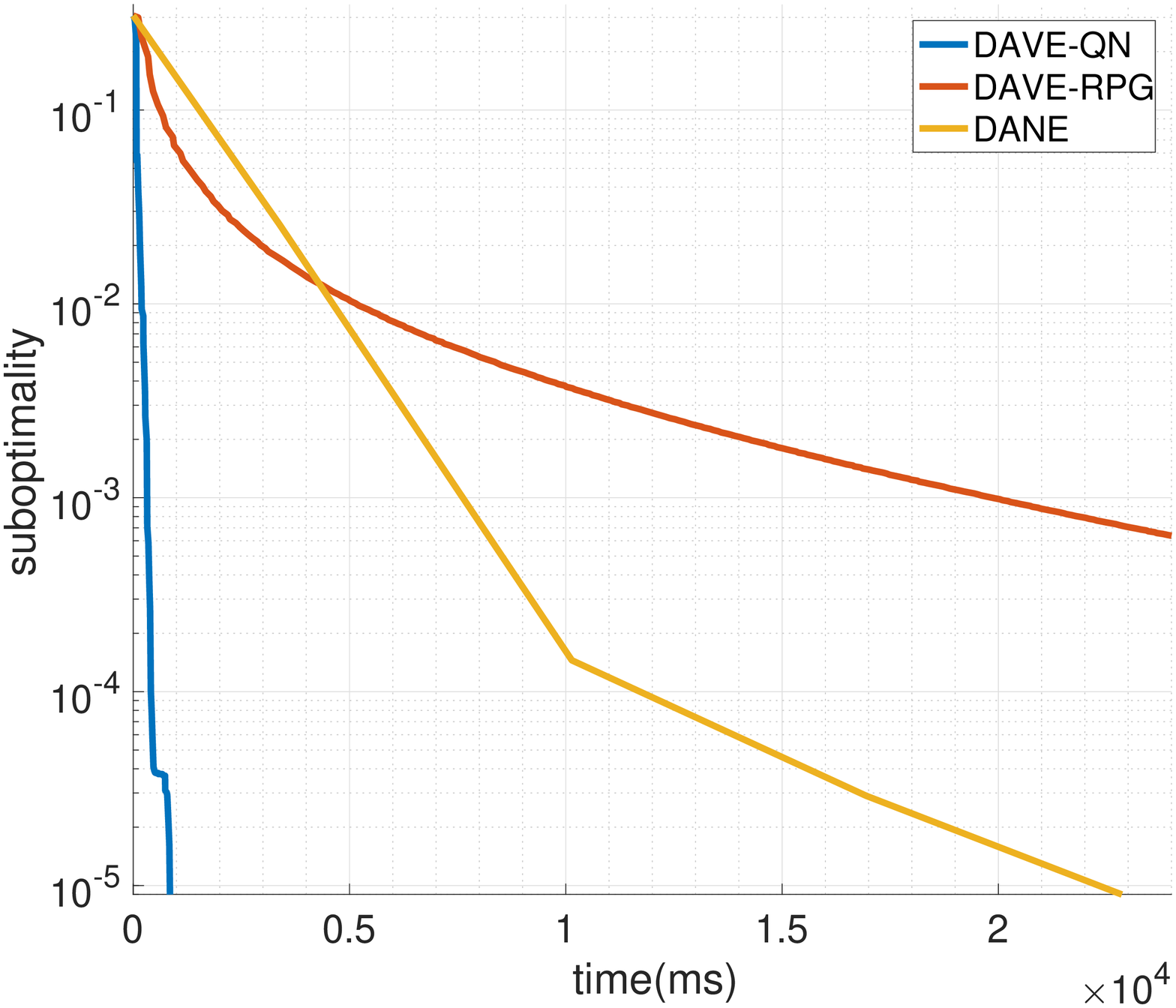} &
\includegraphics[width=0.33\textwidth]{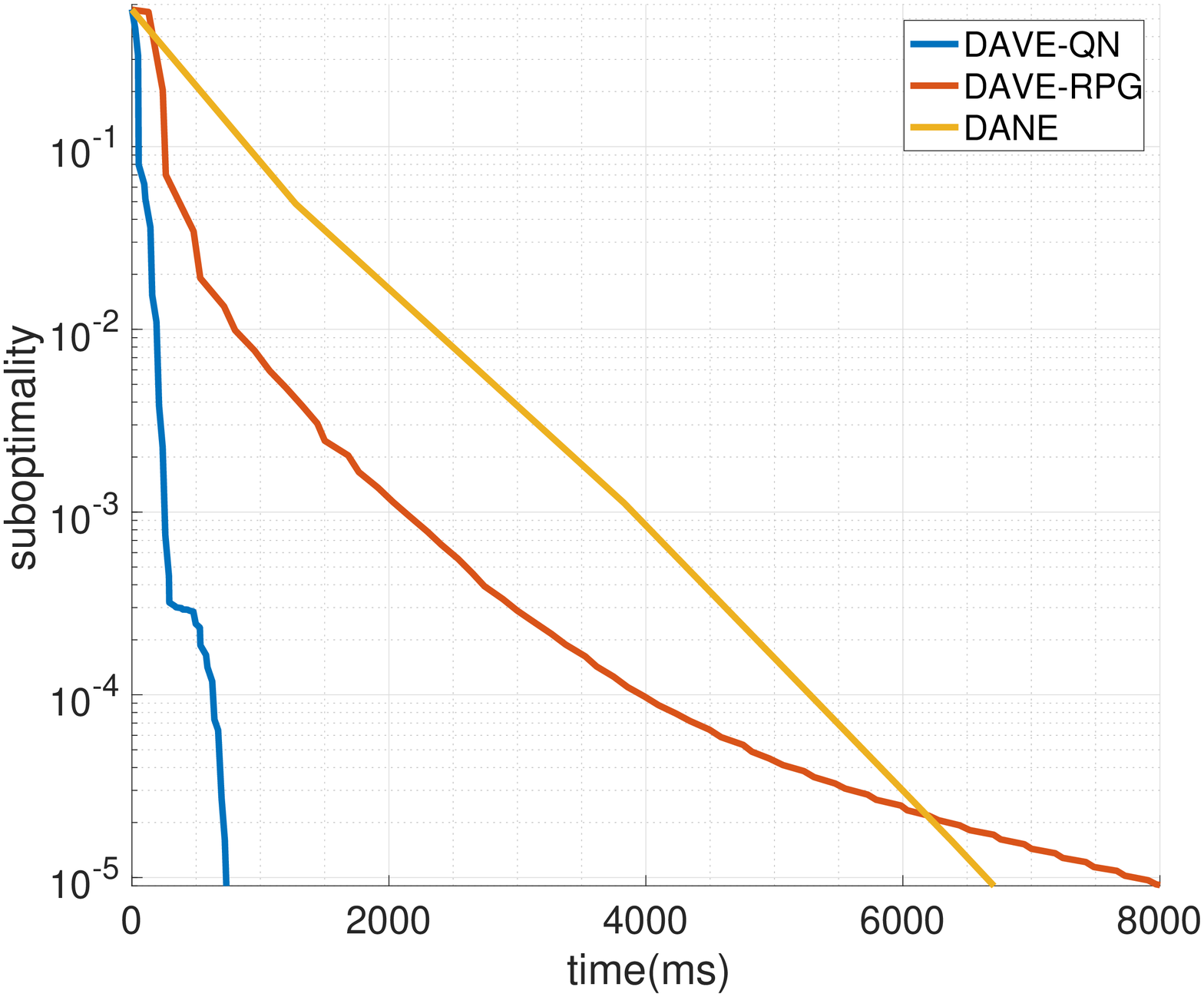} \\
\textbf{SUSY $(p=19,n=5M)$}  & \textbf{covtype $(p=54,n\approx 2.9M)$}  \\[6pt]
\end{tabular}
\caption{Expected suboptimality versus time.}
\label{fig:suboptimality}
\end{figure}

%!TEX root = ex_article_v2.tex
\vspace{-0.2in}
\section{Conclusion and Future Work} 
%\vspace{-0.1in}
In this paper, we focused on the problem of minimizing a large-scale empirical risk minimization in a distributed manner. We used an asynchronous architecture which requires no global coordination between the master node and the workers. Unlike distributed first-order methods that follow the gradient direction to update the iterates, we proposed a distributed averaged quasi-Newton (DAve-QN) algorithm that uses a quasi-Newton approximate Hessian of the workers' local objective function to update the decision variable. In contrast to second-order methods that require computation of the local functions Hessians, the proposed DAve-QN only uses gradient information to improve the convergence of first-order methods in ill-conditioned settings. It is worth mentioning that the computational cost of each iteration of DAve-QN is $\mathcal{O}(p^2)$, while the size of the vectors that are communicated between the master and workers is $\mathcal{O}(p)$. Our theoretical results show that the sequence of iterates generated at the master node by following the update of DAve-QN converges superlinearly to the optimal solution when the objective functions at the workers are smooth and strongly convex. Our results hold for both bounded delays and unbounded delays with a bounded time-average. Numerical experiments illustrate the performance of our method.

The choice of the stepsize in the initial stages of the algorithm is the key to get good overall iteration complexity for second-order methods. Investigating several line search techniques developed for BFGS and adapting it to the distributed asynchronous setting is a future research direction of interest. Another promising direction would be developing Newton-like methods that can go beyond superlinear convergence while preserving communication complexity. Finally, investigating the dependence of the convergence properties on the sample size $m_i$ of each machine $i$ would be interesting, in particular one would expect the performance in terms of communication complexity to improve if the sample size of each machine is increased.

\subsubsection*{Acknowledgments}
This work was supported in part by the grants NSF DMS-1723085, NSF CCF-1814888 and NSF CCF-1657175. This work used the Extreme Science and Engineering Discovery Environment (XSEDE) \cite{towns2014xsede}, which is supported by National Science Foundation grant number ACI-1548562.

{
%\footnotesize{
%\bibliographystyle{alpha}
\newpage
\bibliographystyle{unsrt}
\bibliography{main-dave-qn}
}
%}

\newpage
\appendix
\section{Supplementary Material}\label{alg_with_index}

\subsection{The proposed DAve-QN method with exact time indices}
\tcbset{width=1.0\columnwidth,before=,after=, colframe=black,colback=white, fonttitle=\bfseries, coltitle=white, colbacktitle=red!80!yellow, boxrule=0.2mm}
\begin{algorithm}[H]
\caption{DAve-QN with comprehensive notations)}
\label{alg:DQN}
\centering
\begin{multicols}{2}
\begin{tcolorbox}[title=Master:]
\textbf{Initialize} $ \bbx^0$, $\bbB_i^0$, $(\bbB^0)^{-1} = (\sum_{i=1}^n \bbB_i^0)^{-1}$, $\bbu^0 = \sum_{i=1}^n \bbB_i^0 \bbx^0$, $\bbg^0 = \sum_{i=1}^n\nabla f_i(\bbx^0)$\\
%Send $\overline x$ to each machine\\
\For{$t= 1$ \KwTo T-1}{
    \textbf{If} a worker sends an update:\\
    {\color{blue!70!black} Receive $\Delta \bbu^{t}$, $\bby$, $\bbq$, $\alpha^t$, $\beta^t$ from it}\\
    $\bbu^{t}= \bbu^{t-1} + \Delta \bbu^{t}$ \\
    $\bbg^{t} = \bbg^{t-1} + \bby$\\
%     $\bbB_i^{t} = \bbB_i^{t-1} + \frac{\bby \bby^T}{\alpha} - \frac{\bbz \bbz^T}{\beta}$\\
    $\bbv^t = (\bbB^{t-1})^{-1} \bby$\\
    $\bbU^{t} = (\bbB^{t-1})^{-1} - \frac{\bbv^t \bbv^{t\top}}{\alpha^t + \bbv{t\top}\bby}$\\
    $\bbw^t = \bbU^{t} \bbq$\\
    $\Binv = \bbU^{t} + \frac{\bbw^t \bbw^{t\top}}{\beta^t - \bbq^T\bbw^t}$\\
    $\bbx^{t} = \Binv (\bbu^{t} - \bbg^{t})$\\
    {\color{red!80!yellow} Send $\bbx^{t}$ to the slave in return}
}
Interrupt all slaves\\
\textbf{Output} $ \bbx^T $
\end{tcolorbox}

\columnbreak
\tcbset{width=1.03\columnwidth,before=, colframe=black!50!black, colbacktitle=blue!70!black}
\begin{tcolorbox}[title=Slave $i$:]
\textbf{Initialize} $\bbx_i^0 = \bbx^0$, $\bbB_i^0$\\
\While{not interrupted by master}{
    {\color{red!80!yellow} Receive $\bbx^{t - D_i^t}$ at moment $t - D_i^t$}\\
    Perform below steps by moment $t - d_i^t$ \\
    $\bbz_i^t=\bbz_i^{t-d_i^t} = \bbx^{t - D_i^t}$\\
%     $\bbB_i^t = \bbB_i^{t - D_i^t}$\\
    $\bbs_i^t = \bbs_i^{t-d_i^t} = \bbz^{t - d_i^t} - \bbz_i^{t - D_i^{t}}$\\	
    $\bby_i^t=\bby_i^{t-d_i^t} = \nabla f_i(\bbx^{t-d_i^t}) - \nabla f_i(\bbz_i^{t})$\\
    $\bbq_i^t= \bbq_i^{t-d_i^t} = \bbB_i^{t-D_i^t} \bbs_i^{t-d_i^t}$\\
    $\alpha^{t-d_i^t} = \bby_i^{tT} \bbs_i^{t}$\\
    $\beta^{t-d_i^t} = (\bbs_i^{t-d_i^t})^T \bbB_i^{t-D_i^t} \bbs_i^t$\\
    $\bbB_i^{t-d_i^t} = \bbB_i^{t - D_i^t} + \frac{\bby_i^{t} \bby_i^{tT}}{\alpha^{t-d_i^t}} - \frac{\bbq_i^{t} \bbq_i^{tT}}{\beta^{t-d_i^t}}$\\
    $\Delta \bbu^{t-d_i^t} = \bbB_i^{t-d_i^t}\bbz_i^{t-d_i^t}  - \bbB_i^{t-D_i^t}\bbz_i^{t-D_i^t}$\\
    {\color{blue!70!black} Send $\Delta \bbu^{t-d_i^t}, \bby_i^{t-d_i^t}, \bbq_i^{t-d_i^t}, \alpha^{t-d_i^t}, \beta^{t-d_i^t}$ to the master at moment $t-d_i^t$}
}
\end{tcolorbox}
\end{multicols}
\end{algorithm}
\subsection{Proof of Lemma \ref{lemm_update}} 
\begin{proof}
	To verify the claim, we need to show that $\bbu^t = \sum_{i=1}^n \bbB_i^t \bbz_i^t$ and $\bbg^t = \sum_{i=1}^n \nabla f_i(\bbz_i^t)$. They follow from our delayed vectors notation $\bbz_i^t = \bbz_i^{t-d_i^t}$ and how $\Delta \bbu^{t-d_i^t}$ and $\bby_i^{t-d_i^t}$ are computed by the corresponding worker.
\end{proof}
\subsection{Proof of Lemma~\ref{lem:M_2}}
To prove the claim in Lemma~\ref{lem:M_2} we first prove the following intermediate lemma using the result of Lemma 5.2 in \cite{Broyden}.

%%%%%%%%%%%%%%%%%%%%%%%%%%%%%%%%%%
%%%%%%%%%%%%%%%%%%%%%%%%%%%%%%%%%%
%%%%%    L   E   M   M   A       %%%%%%%%%%%%%%%%%%
%%%%%%%%%%%%%%%%%%%%%%%%%%%%%%%%%%
%%%%%%%%%%%%%%%%%%%%%%%%%%%%%%%%%%
\begin{lemma} \label{lem:M_1}
Consider the proposed method outlined in Algorithm \ref{alg:DQN}.
Let $\bbM$ be a nonsingular symmetric matrix such that 
\begin{equation}\label{ness_cond}
\|\bbM\bby_i^t-\bbM^{-1}\bbs_i^t\| \leq \beta \|\bbM^{-1}\bbs_i^t\| ,
\end{equation}
for some $\beta \in[0,1/3]$ and vectors $\bbs_i^t$ and $\bby_i^t$ in $\reals^p$ with $\bbs_i^t\neq\bb0$. Let's denote $i$ as the index that has been updated at time $t$. Then, there exist positive constants $\alpha$, $\alpha_1$, and $\alpha_2$ such that, for any symmetric $\bbA\in \reals^{p\times p}$ we have, 
\begin{align}\label{B_cons_rela}
\|\bbB_i^{t}-\bbA\|_\bbM 
	&\leq \left[ (1- \alpha \theta^2)^{1/2} 
	+\alpha_1 \frac{\|\bbM\bby_i^{t-D_i^t}-\bbM^{-1}\bbs_i^{t-D_i^t}\| }{\|\bbM^{-1}\bbs_i^{t - D_i^t}\|} \right] \|\bbB_i^{t - D_i^t}-\bbA\|_\bbM \nonumber\\
    &\quad + \alpha_2 \frac{\|\bby_i^{t - D_i^t}-\bbA\bbs_i^{t - D_i^t}\|}{\|\bbM^{-1}\bbs_i^{t - D_i^t}\|},
\end{align}
where $\alpha= (1-2\beta)/(1-\beta^2)\in [3/8,1]$, $\alpha_1=2.5(1-\beta)^{-1}$, $\alpha_2=2(1+2\sqrt{p})\|\bbM\|_\bbF$, and
%%%%
\begin{align}
\theta =\frac{\|\bbM(\bbB_i^{t - D_i^t}-\bbA)\bbs_i^{t - D_i^t}\|}{\|\bbB_i^{t - D_i^t}-\bbA\|_\bbM\|\bbM^{-1}\bbs_i^{t - D_i^t}\|} \quad \for \ \bbB_i^{t - D_i^t}\neq \bbA, \qquad  \theta=0\quad  \for\   \bbB_i^{t - D_i^t} = \bbA.
\end{align}
%%%
\end{lemma}

\begin{proof}
By definition of delays $d_i^t$, the function $f_i$ was updated at step $t - d_i^t$ and $\bbB_i^{t-1}$ is equal to $\bbB_i^{t-D_i^t}$. Considering this observation and the result of Lemma 5.2 in \cite{Broyden}, the claim follows.
\end{proof}

Note that the result in Lemma~\ref{lem:M_1} characterizes an upper bound on the difference between the Hessian approximation matrices $\bbB_i^{t }$ and $\bbB_i^{t - D_i^t}$ and any positive definite matrix $\bbA$. Let us show that matrices $\bbM=\nabla^2f_i(\bbx^*)^{-1/2}$ and $\bbA=\nabla^2f_i(\bbx^*)$ satisfy the conditions of Lemma \ref{lem:M_1}. By strong convexity of $f_i$ we have $\|\nabla^2 f_i(\bbx^*)^{1/2}\bbs_i^{t - D_i^t}\| \ge \sqrt{\mu}\|\bbs_i^{t - D_i^t}\|$. Combined with Assumption \ref{ass:hessian_assumption}, it gives that 
\begin{align}\label{proof10}
 \frac{\| \bby_i^{t - D_i^t}-\nabla^2f_i(\bbx^*)\bbs_i^{t - D_i^t}\|}{\|\nabla^2f_i(\bbx^*)^{1/2}\bbs_i^{t - D_i^t}\|}
\leq \frac{ \tilde{L}\|\bbs_i^{t - D_i^t}\|\max\{\|\bbz_i^{t - D_i^t}-\bbx^*\|,\|\bbz_i^{t}-\bbx^*\|\} }{\sqrt{\mu}\|\bbs_i^{t - D_i^t}\|} = \frac{\tilde{L}}{\sqrt{\mu}} \sigma_i^{t}
\end{align}
%%%
This observation implies that the left hand side of the condition in \eqref{ness_cond} for $\bbM=\nabla^2f_i(\bbx^*)^{-1/2}$ is bounded above by 
\begin{align}\label{proof20}
\frac{\|\bbM\bby_i^{t - D_i^t}-\bbM^{-1}\bbs_i^{t - D_i^t}\|}{\|\bbM^{-1}\bbs_i^{t - D_i^t}\|}
\leq
\frac{\| \nabla^2f_i(\bbx^*)^{-1/2}\| \|\bby_i^{t - D_i^t}-\nabla^2f_i(\bbx^*)\bbs_i^{t - D_i^t}\|}{\|\nabla^2f_i(\bbx^*)^{1/2}\bbs_i^{t - D_i^t}\|}
 \leq \frac{\tilde{L}}{\mu} \sigma_i^{t}
\end{align}
%%%
Thus, the condition in \eqref{ness_cond} is satisfied since $\tilde{L} \sigma_i^{t}/\mu<1/3.$ Replacing the upper bounds in \eqref{proof10} and \eqref{proof20} into the expression in \eqref{B_cons_rela} implies the claim in \eqref{eq:lemma_M_2_claim} with
\begin{align}\label{proof300}
\beta= \frac{\tilde{L}}{\mu} \sigma_i^{t}, \ 
 \alpha=\frac{1-2\beta}{1-\beta^2}, \ 
\alpha_3=\frac{5\tilde{L}}{2\mu(1-\beta)}, \ 
\alpha_4=    \frac{{2(1+2\sqrt{p}) \tilde{L}}}{\sqrt{\mu}}\|\nabla^2f_i(\bbx^*)^{-\frac{1}{2}}\|_\bbF,
\end{align}
%%%
and the proof is complete.

\subsection{Proof of Lemma~\ref{lemma:lin_convergence}}

We first state the following result from Lemma 6 in \cite{mokhtari2018iqn}, which shows an upper bound for the error $\|\bbx^{t}-\bbx^*\|$ in terms of the gap between the delayed variables $\bbz_i^t$ and the optimal solution $\bbx^*$ and the difference between the Newton direction $\nabla^2 f_{i}(\bbx^*)\left(\bbz_i^t-\bbx^*\right)$ and the proposed quasi-Newton direction $\bbB_i^t\left(\bbz_i^t-\bbx^*\right)$.

%%%%%%%%%%%%%%%%%%%%%%%%%%%%%%%%%%
%%%%%%%%%%%%%%%%%%%%%%%%%%%%%%%%%%
%%%%%    L   E   M   M   A       %%%%%%%%%%%%%%%%%%
%%%%%%%%%%%%%%%%%%%%%%%%%%%%%%%%%%
%%%%%%%%%%%%%%%%%%%%%%%%%%%%%%%%%%
\begin{lemma} \label{lem:recurrence_to_average}
If Assumptions \ref{ass:gradient_assumption} and \ref{ass:hessian_assumption} hold, then the sequence of iterates generated by Algorithm \ref{alg:DQN} satisfies
%%%
\begin{align}\label{eq:recurrence_to_average}
&\|\bbx^{t}-\bbx^*\|
%\\&
\leq
 \frac{\tilde{L}\Gamma^{t}}{n} \sum_{i=1}^n  \left\|\bbz_i^t-\bbx^* \right\|^2
 +
 \frac{\Gamma^{t}}{n} \sum_{i=1}^n  \left\| \left(\bbB_i^t-\nabla^2 f_{i}(\bbx^*)\right) \left(\bbz_i^t-\bbx^*\right) \right\|, 
\end{align}
where $\Gamma^t:=\|( (1/n)\sum_{i=1}^n\bbB_i^t )^{-1}\|$.
\end{lemma}

We use the result in Lemma \ref{lem:recurrence_to_average} to prove the claim of Lemma~\ref{lemma:lin_convergence}. We will prove the claimed convergence rate in Lemma \ref{lemma:lin_convergence}%\eqref{eq:linear_conv_for_arbitrary_r} 
together with an additional claim
	\begin{align*}
		\left\|\bbB_i^{t}- \nabla^2 f_i(\bbx^*)\right\|_\bbM \le 2\delta
	\end{align*}		
	by inductions on $m$ and on $t\in [T_m, T_{m+1})$. The base case of our induction is $m=0$ and $t=0$, which is the initialization step, so let us start with it.
	
	Since all norms in finite dimensional spaces are equivalent, there  exists a constant $\eta>0$ such that $\|\bbA\|\leq \eta \|\bbA\|_\bbM$ for all $\bbA$. Define $\gamma\coloneqq 1 /\mu$ and $d\coloneqq \max_m (T_{m+1} - T_m)$, and assume that $\eps(r)=\eps$ and $\delta(r)=\delta$ are chosen such that 
%%%%
\begin{equation} \label{eq:choice_of_delta_eps}
(2\alpha_3\delta+\alpha_4)\frac{d\eps}{1-r}\leq \delta
 \quad \text{and}\quad
 \gamma(1+r)[\tilde{L}\eps+2\eta\delta] \leq r,
\end{equation}
where $\alpha_3$ and $\alpha_4$ are the constants from Lemma \ref{lem:M_2}. As $\|\bbB_i^0-\nabla^2 f_i(\bbx^*)\|_\bbM\leq \delta$, we also have
\[ \|\bbB_i^0-\nabla^2 f_i(\bbx^*)\|\leq \eta \delta.
\]
Therefore, by triangle inequality from $\|\nabla^2 f_i(\bbx^*)\|\leq L$ we obtain $\|\bbB_i^0\|\leq \eta\delta+L$, so $\|(1/n)\sum_{i=1}^n\bbB_i^0\|\leq \eta\delta+L$. The second part of inequality \eqref{eq:choice_of_delta_eps} also implies $2\gamma(1+r)\eta\delta\leq r$. 
Moreover, it holds that $\|\bbB_i^0-\nabla^2 f_i(\bbx^*)\|\leq \eta \delta<2 \eta \delta$ and  by Assumption \ref{ass:gradient_assumption} $\gamma\geq\|\nabla^2 f_i(\bbx^*)^{-1}\|$, so we obtain by Banach Lemma that
%By Banach lemma, if a matrix $\bbA$ has norm $\|\bbA\|< 1$, then $\|(\bbI  + \bbA)^{-1}\|\le (1 - \|\bbA\|)^{-1}$, so
\begin{eqnarray*}
	\|(\bbB_i^0)^{-1}\| \leq (1+r)\gamma.
\end{eqnarray*}
We formally prove this result in the following lemma. 

\begin{lemma}
If the Hessian approximation $\bbB_i$ satisfies the inequality $\|\bbB_i - \nabla^2 f_i(\bbx^*)\| \leq 2\eta \delta $ and $\|\nabla^2 f_i(\bbx^*)^{-1}\|\leq \gamma$, then we have $\|\bbB_i^{-1}\| \leq (1+r)\gamma.$
\end{lemma}

\begin{proof}
Note that according to Banach Lemma, if a matrix $\bbA$ satisfies the inequality $\|\bbA-\bbI\| \leq 1$, then it holds $\|\bbA^{-1}\| \leq \frac{1}{1-\|\bbA-\bbI\|}$.

We first show that $\|\nabla^2 f_i(\bbx^*)^{-1/2}  \bbB_i \nabla^2 f_i(\bbx^*)^{-1/2}-\bbI\|\leq 1$. To do so, note that 
%%%
\begin{align}
\|\nabla^2 f_i(\bbx^*)^{-1/2}  \bbB_i \nabla^2 f_i(\bbx^*)^{-1/2} - \bbI\|   
&\leq   
\|\nabla^2 f_i(\bbx^*)^{-1/2}\| \|  \bbB_i  - \nabla^2 f_i(\bbx^*)\| \|\nabla^2 f_i(\bbx^*)^{-1/2}\|  \nonumber\\
&\leq    2\eta \delta \gamma   \nonumber\\
&\leq    \frac{r}{r+1}  \nonumber\\
&<1 .
\end{align}
%%%%
Now using this result and Banach Lemma we can show that 
%%%
\begin{align}
\|\nabla^2 f_i(\bbx^*)^{1/2}  \bbB_i^{-1} \nabla^2 f_i(\bbx^*)^{1/2}\| 
&\leq \frac{1}{1-\|\nabla^2 f_i(\bbx^*)^{-1/2}  \bbB_i \nabla^2 f_i(\bbx^*)^{-1/2} - \bbI\|   }\nonumber\\
&\leq \frac{1}{1-\frac{r}{r+1}  }\nonumber\\
&= 1+r
\end{align}
%%%
Further, we know that 
%%%
\begin{align}
\|\nabla^2 f_i(\bbx^*)^{1/2}  \bbB_i^{-1} \nabla^2 f_i(\bbx^*)^{1/2}\| \geq \frac{\|\bbB_i^{-1} \|}{\gamma}
\end{align}
By combining these results we obtain that 
%%%
\begin{equation}
\|\bbB_i^{-1} \| \leq (1+r)\gamma.
\end{equation}
\end{proof}

Similarly, for matrix $((1/n)\sum_{i=1}^n \bbB_i^0)^{-1}$ we get from $\|(1/n)\sum_{i=1}^n \bbB_i^0-(1/n)\sum_{i=1}^n \nabla^2 f_i(\bbx^*)\| 
\leq (1/n)\sum_{i=1}^n\| \bbB_i^0- \nabla^2 f_i(\bbx^*)\| \leq\eta \delta$ and  $\|\nabla^2 f(\bbx^*)^{-1}\|\leq \gamma$ that
%%%%
\begin{align*}
\left\|\left(\frac{1}{n}\sum_{i=1}^n \bbB_i^0\right)^{-1}\right\|\leq (1+r)\gamma.
\end{align*}
We have by Lemma \ref{lem:M_2} and induction hypothesis
\begin{align*}
\left\|\bbB_i^{t}- \nabla^2 f_i(\bbx^*)\right\|_\bbM - \left\|\bbB_i^{t-D_i^t}- \nabla^2 f_i(\bbx^*)\right\|_\bbM
& \leq 
	\alpha_3 \sigma_i^{t-D_i^t} \left\|\bbB_i^{t-D_i^t}- \nabla^2 f_i(\bbx^*)\right\|_\bbM  + \alpha_4 \sigma_i^{t-D_i^t} \\
	&\le \left(\alpha_3  \left\|\bbB_i^{t-D_i^t}- \nabla^2 f_i(\bbx^*)\right\|_\bbM  + \alpha_4\right)  r^{m-1}\epsilon\\
	&\le \left(2\alpha_3\delta +  \alpha_4\right)  r^{m-1}\epsilon,
\end{align*}
By summing this inequality over all moments in the current epoch when worker $i$ performed its update, we obtain that
\begin{align*}
\left\|\bbB_i^{t}- \nabla^2 f_i(\bbx^*)\right\|_\bbM - \left\|\bbB_i^{T_m - d_i^{T_m}}- \nabla^2 f_i(\bbx^*)\right\|_\bbM
	 \le\left(2\alpha_3\delta +  \alpha_4\right)  dr^{m - 1}\epsilon ,
\end{align*}
Summing the new bound again, but this time over all passed epoch, we obtain
\begin{align*}
\left\|\bbB_i^{t}- \nabla^2 f_i(\bbx^*)\right\|_\bbM - \left\|\bbB_i^{0}- \nabla^2 f_i(\bbx^*)\right\|_\bbM
	\le\left(2\alpha_3\delta +  \alpha_4\right)  d\epsilon \sum_{k=0}^{m-1}r^{k} 
	\le\frac{\left(2\alpha_3\delta +  \alpha_4\right)  d\epsilon}{1-r}\le \delta.
\end{align*}
%%%
Therefore, $\left\|\bbB_i^{t}- \nabla^2 f_i(\bbx^*)\right\|_\bbM\leq 2\delta$. By using the Banach argument again, we can show that $\|(\frac{1}{n}\sum_{i=1}^n \bbB_i^t)^{-1}\|\leq (1+r)\gamma$. Using this result, for any $t\in [T_m, T_{m+1})$ we have $z_i^t = x^{t - D_i^t}\in [T_{m-1}, t)$ and we can write
\begin{align}
\|\bbx^{t}-\bbx^*\|& \leq   (1+r)\gamma \left[  \frac{\tilde{L}}{n} \sum_{i=1}^n  \left\|\bbz_i^t-\bbx^* \right\|^2+  \frac{1}{n} \sum_{i=1}^n  \left\| \left[\bbB_i^t-\nabla^2 f_{i}(\bbx^*)\right] \left(\bbz_i^t-\bbx^*\right) \right\|  \right]\nonumber\\
&\leq (1+r)\gamma\left[ {\tilde{L}} \eps +2\eta\delta \right] 
\max_i \|\bbz_i^t - \bbx^*\|\nonumber\\
&\leq r\max_i \|\bbz_i^t - \bbx^*\| \nonumber \\
&\leq r^m\|\bbx^0 - \bbx^*\| .
\end{align}

\subsection{Proof of Theorem \ref{thm:sup_linear_thm}}

	Dividing both sides of \eqref{eq:recurrence_to_average} by $(1/n) \sum_{i=1}^n  \left\|\bbz_i^t-\bbx^* \right\|$, we get
%%%
\begin{equation}\label{eq:lemma_claim_200}
\frac{\|\bbx^{t}-\bbx^*\|}{ \frac{1}{n}\sum_{i=1}^n  \left\|\bbz_i^t-\bbx^* \right\|}
\leq{\tilde{L}\Gamma^{t}} \sum_{i=1}^n \frac{  \left\|\bbz_i^t-\bbx^* \right\|^2}{ \sum_{i=1}^n  \left\|\bbz_i^t-\bbx^* \right\|}
 +
 {\Gamma^{t}}\sum_{i=1}^n \frac{  \left\| \left(\bbB_i^t-\nabla^2 f_{i}(\bbx^*)\right) \left(\bbz_i^t-\bbx^*\right) \right\|}{ \sum_{i=1}^n  \left\|\bbz_i^t-\bbx^* \right\|}
\end{equation}
 %%%%
As every term in $\sum_{i=1}^n  \|\bbz_i^t-\bbx^* \|$ is non-negative, the upper bound in \eqref{eq:lemma_claim_200} will remain valid if we keep only one summand out of the whole sum in the denominators of the right-hand side, so
\begin{align}
 \frac{\|\bbx^{t}-\bbx^*\|}{ \frac{1}{n}\sum_{i=1}^n  \left\|\bbz_i^t-\bbx^* \right\|}
%\\&
& \leq
 {\tilde{L}\Gamma^{t}} \sum_{i=1}^n \frac{  \left\|\bbz_i^t-\bbx^* \right\|^2}{   \left\|\bbz_i^t-\bbx^* \right\|}
 +
 {\Gamma^{t}} \sum_{i=1}^n \frac{  \left\| \left(\bbB_i^t-\nabla^2 f_{i}(\bbx^*)\right) \left(\bbz_i^t-\bbx^*\right) \right\|}{  \left\|\bbz_i^t-\bbx^* \right\|}
  \nonumber\\
  &
  =
  {\tilde{L}\Gamma^{t}} \sum_{i=1}^n {  \left\|\bbz_i^t-\bbx^* \right\|}
   +
 {\Gamma^{t}} \sum_{i=1}^n \frac{  \left\| \left(\bbB_i^t-\nabla^2 f_{i}(\bbx^*)\right) \left(\bbz_i^t-\bbx^*\right) \right\|}{  \left\|\bbz_i^t-\bbx^* \right\|}. 	\label{eq:recurrence_for_limit}
\end{align}

Now using the result in Lemma 5 of \cite{mokhtari2018iqn}, the second sum in \eqref{eq:recurrence_for_limit} converges to zero. Further, $\Gamma^t$ is bounded above by a positive constant. Hence, by computing the limit of both sides in \eqref{eq:recurrence_for_limit} we obtain
\begin{align*}
\lim_{t\to\infty}\frac{\|\bbx^{t}-\bbx^*\|}{ \frac{1}{n}\sum_{i=1}^n  \left\|\bbz_i^t-\bbx^* \right\|}  = 0.
\end{align*}
%%%
Therefore, if $T$ is big enough, for $t>T$ we have
\begin{align}\label{eq:recurrence_with_pure_average}
	\|\bbx^{t}-\bbx^*\| \le \frac{1}{n}\sum_{i=1}^n  \left\|\bbz_i^t-\bbx^* \right\| 
    = \frac{1}{n}\sum_{i=1}^n  \left\|\bbx_i^{t-D_i^t}-\bbx^* \right\|
    \le \max_i \left\|\bbx_i^{t-D_i^t}-\bbx^* \right\|.
\end{align}
Now, let $t_0=t_0(m)\coloneqq \min\{\tilde t\in [T_{m+1}, T_{m+2}): \|\bbx^{\tilde t} - \bbx^*\| = \max_{t \in [T_{m+1}, T_{m+2})} \|\bbx^t - \bbx^*\| \}$. In other words, $t_0$ is the first moment in epoch $m+1$ attaining the maximal distance from $x^*$. Then, for all $t\in [T_{m+1}, t_0)$ we have $\|\bbx^t - \bbx^*\| < \|\bbx^{t_0} - \bbx^*\|$. Furthermore, from equation \eqref{eq:recurrence_with_pure_average} and the fact that, according to Proposition~\ref{pr:delays_recurrence}, $t_0 - D_i^{t_0}\in [T_m, t_0)$ we get
\begin{align*}
	\max_{t\in [T_{m+1}, T_{m+2})}\| \bbx^{t} - \bbx^*\| = \| \bbx^{t_0} - \bbx^*\| \le \max_i\left\|\bbx_i^{t_0-D_i^{t_0}}-\bbx^* \right\| 
    \le \max_{t\in [T_{m}, t_0)}\| \bbx^{t} - \bbx^*\|.
\end{align*}
Note that it can not happen that $\max_{t\in [T_{m}, t_0)}\| \bbx^{t} - \bbx^*\| = \max_{t\in [T_{m+1}, t_0)}\| \bbx^{t} - \bbx^*\|$ as that would mean that there exists a $\hat t\in [T_{m+1}, t_0)$ such that $\|\bbx^{\hat t} - \bbx^*\| \ge \|\bbx^{t_0} - \bbx^*\|$, which we made impossible when defining $t_0$. Then, the only option is that in fact
\begin{align*}
	\max_{t\in [T_{m}, t_0)}\| \bbx^{t} - \bbx^*\|
    =\max_{t\in [T_{m}, T_{m+1})}\| \bbx^{t} - \bbx^*\|.
\end{align*}
Finally,
\begin{align*}
	\lim_{t \rightarrow \infty} \frac{ \max_{t\in [T_{m+1}, T_{m+2})}\| \bbx^{t} - \bbx^*\|}{\max_{t \in [T_m, T_{m+1})} \|\bbx^t - \bbx^*\|} 
    &= \lim_{t \rightarrow \infty} \frac{\| \bbx^{t_0(m)} - \bbx^*\|}{\max_{t \in [T_m, T_{m+1})} \|\bbx^t - \bbx^*\|} \\
    &= \lim_{t \rightarrow \infty} \frac{\| \bbx^{t_0(m)} - \bbx^*\|}{\max_{t \in [T_m, t_0(m))} \|\bbx^t - \bbx^*\|} \\
    &\le \lim_{t \rightarrow \infty} \frac{\| \bbx^{t_0(m)} - \bbx^*\|}{\max_i \|\bbx^{t_0(m) - D_i^{t_0(m)}} - \bbx^*\|}\\
    &\le \lim_{t \rightarrow \infty} \frac{\| \bbx^{t_0(m)} - \bbx^*\|}{\frac{1}{n}\sum_{i=1}^n \|\bbz_i^{t_0(m)} - \bbx^*\|} = 0,
\end{align*}
where at the last step we used again the fact that $\bbz_i^t = \bbx^{t - D_i^t}$.

\end{document}